\documentclass[12pt]{amsart}
\usepackage{graphicx}
\usepackage[centertags]{amsmath}
\usepackage{amsfonts}
\usepackage{amssymb}
\usepackage{amsthm}
\usepackage{newlfont}
\newtheorem{thm}{Theorem}[section]

\newtheorem{prop}[thm]{Proposition}
\theoremstyle{definition}
\newtheorem{defn}[thm]{Definition}
\theoremstyle{remark}

\newcommand{\grad}{\mathrm{grad}}

\newcommand{\id}{\mathrm{id}}
\title{Four-dimensional Riemannian manifolds with two circulant structures}
\bigskip
\author{Dimitar Razpopov}
\date{}
\begin{document}

\maketitle

%%%%%%%%%%%%%%%%%%%%%%%%%%%%%%%%%
\begin{abstract}
%%%%%%%%%%%%%%%%%%%%%%%%%%%%%%%%%%%
We consider a class $(M, g, q)$ of four-dimensional Riemannian manifolds $M$, where beside the metric $g$ there is an additional structure $q$, whose fourth power is the unit matrix. We use the existence of a local coordinate system such that there the coordinates of $g$ and $q$ are circulant matrices. In this system $q$ has constant coordinates and $q$ is an isometry with respect to $g$. By the special identity for the curvature tensor $R$ generated by the connection $\nabla$ of $g$ we define a subclass of $(M, g, q)$. For any $(M, g, q)$ in this subclass we get some assertions for the sectional curvatures of two-planes. We get the necessary and sufficient condition for $g$ such that $q$ is parallel with respect to $\nabla$.

\end{abstract}

\Small{\textbf{Mathematics Subject Classification (2010)}: 53C15,
53B20}

\Small{\textbf{Keywords}: Riemannian metric,
curvature properties, circulant matrix}
% ----------------------------------------------------------------
\thispagestyle{empty}
\section{Introduction}
The main purpose of the present paper is to continue the
considerations on some Riemannian manifolds using the existing of an useful local circulant coordinate system analogously to \cite{21}, \cite{3}, \cite{4}.

 In Sect.~\ref{1} we introduce four-dimensional differentiable manifold $M$ with a Riemannian metric $g$ whose matrix in local coordinates is a special circulant matrix. Furthermore, we consider an additional structure $q$ on $M$ with $q^{4}=\id$ such that its matrix in local coordinates is also circulant. Thus, the structure $q$ is an isometry with respect to $g$.
We denote by $(M, g, q)$ the manifold $M$ equipped with the metric $g$ and the structure $q$.
In Sect. \ref{2} in Theorem~\ref{th1} we obtain that an orthogonal basis of type $\{x, qx, q^{2}x, q^{3}x\}$ exists in the tangent space of a manifold $(M, g, q)$. In Sect.~\ref{3} we establish relations between the sectional curvatures of some special $2$-planes in the tangent space. In Sect.~\ref{4} we obtain a necessary and sufficient condition for $q$ to be parallel with respect to the Riemannian connection of $g$.

\section{Preliminaries}\label{1}
Let $M$ be a four-dimensional manifold with a Riemannian metric $g$. Let the local components of the metric $g$ at an arbitrary point $p(X^{1}, X^{2}, X^{3}, X^{4})\in M$ form the following circulant matrix:
 \begin{equation}\label{f1}
    (g_{ij})=\begin{pmatrix}
      A & B & C & B \\
      B & A & B & C\\
      C & B & A & B\\
      B & C & B & A\\
    \end{pmatrix},\qquad
\end{equation}
where $A=A(p), B=B(p), C=C(p)$ are smooth functions.

We suppose
\begin{equation}\label{ab}
  0 < B < C < A\ .
\end{equation}
 Then the conditions to be a positive definite metric $g$
are satisfied:
\begin{equation*}
    A>0,\quad \begin{vmatrix}
      A & B\\
      B & A \\
    \end{vmatrix}=(A-B)(A+B)>0,
\end{equation*}
\begin{equation*}
     \begin{vmatrix}
      A & B & C \\
      B & A & B \\
      C & B & A \\
    \end{vmatrix}=(A-C)\Big(A(C+A)-2B^{2}\Big)>0,
\end{equation*}
\begin{equation*}
     \begin{vmatrix}
     A & B & C & B \\
      B & A & B & C \\
      C & B & A & B\\
      B & C & B & A\\
    \end{vmatrix}= (A-C)^{2}\Big((A+C)^{2}-4B^{2}\Big)>0.
\end{equation*}

We denote by $(M, g)$ the manifold $M$ equipped with the Riemannian metric $g$ defined by (\ref{f1}) with conditions (\ref{ab}).

Let $q$ be an endomorphism in the tangent space $T_{p}M$ of the manifold $(M,g)$. We suppose the local coordinates of $q$ are given by the circulant matrix
\begin{equation}\label{f2}
    (q_{i}^{s})=\begin{pmatrix}
      0 & 1 & 0 & 0\\
      0 & 0 & 1 & 0\\
      0 & 0 & 0 & 1\\
      1 & 0 & 0 & 0\\
    \end{pmatrix}.
\end{equation}
Then $q$ satisfies
\begin{equation}\label{1.3}
    q^{4}=\id,\qquad q^{2}\neq \pm \id.
\end{equation}

The manifold $(M, g)$ equipped with the structure $q$, defined by (\ref{f2}), we denote by $(M, g, q)$.

Further, $x, y, z, u$ will stand for arbitrary elements of the algebra on the smooth vector fields on $M$ or vectors in the tangent space $T_{p}M$. The Einstein summation convention is used, the range of the summation indices being always $\{1, 2, 3, 4\}$.

From (\ref{f1}) and (\ref{f2}) we get immediately the following
\begin{thm} The structure $q$ of the manifold $(M, g, q)$ is an isometry with respect to the metric $g$, i.e.
\begin{equation}\label{2.2}
    g(qx, qy)=g(x, y).
\end{equation}
\end{thm}

\section{Orthogonal $q$-bases of $T_{p}M$}\label{2}
\begin{defn}
A basis of type $\{x, qx, q^{2}x, q^{3}x\}$ of $T_{p}M$ is called a $q$-\textit{basis}. In this case we say that \textit{the vector $x$ induces a $q$-basis of} $T_{p}M$.
\end{defn}

Obviously, we have the following
\begin{prop}
A vector $x=(x^{1},x^{2},x^{3}, x^{4})$ induces a $q$-basis of $T_{p}M$ if and only if
   \begin{equation}\label{lema2}
   \Big((x^{1}-x^{3})^{2}+(x^{2}-x^{4})^{2}\Big)\Big((x^{1}+x^{3})^{2}-(x^{2}+x^{4})^{2}\Big)\neq 0
   \end{equation}
\end{prop}
\begin{proof}
If $x=(x^{1},x^{2},x^{3}, x^{4})\in T_{p}M$, then $qx=(x^{2},x^{3}, x^{4}, x^{1})$,\\ $q^{2}x=(x^{3},x^{4},x^{1}, x^{2})$, $q^{3}x=(x^{4},x^{1},x^{2}, x^{3})$.
The determinant of coordinates of the vectors $x, qx, q^{2}x, q^{3}x$ is just the left side of \eqref{lema2}.
  The vectors $x, qx, q^{2}x, q^{3}x$ are linearly independent which imply \eqref{lema2}.
\end{proof}

\begin{thm}\label{t3.1}
If $x=(x^{1},x^{2},x^{3}, x^{4})$ induces a $q$-basis of $T_{p}M$, then for the angles
 $\angle(x,qx),\ \angle(x,q^{2}x)$, $\angle(qx,q^{2}x)$, $\angle(qx,q^{3}x)$, $\angle(x,q^{3}x)$ and $\angle(q^{2}x,q^{3}x)$
 we have
 \begin{equation*}
   \angle(x,qx)=\angle(qx,q^{2}x)=\angle(x,q^{3}x)=\angle(q^{2}x,q^{3}x),\quad \angle(x,q^{2}x)=\angle(qx,q^{3}x).
\end{equation*}
 \end{thm}

\begin{proof}

Evidently from (\ref{2.2}) we have $g(q^{3}x, q^{3}y)=g(q^{2}x, q^{2}y)=g(qx, qy)=g(x, y)$.
Then from the well known formula $$\cos\angle (x, y) =\dfrac{g(x, y)}{\sqrt{g(x, x)}\sqrt{g(y, y)}}$$ we get $\cos\angle(x,qx)=\cos\angle(qx,q^{2}x)=\cos\angle(x,q^{3}x)=\cos\angle(q^{2}x,q^{3}x)$ and \\ $\cos\angle(x,q^{2}x)=\cos\angle(qx,q^{3}x)$.
\end{proof}
%%% ----------------------------------------------------------------

\begin{thm}\label{th1}
Let $x$ induce a $q$-basis in $T_{p}M$ of a manifold $(M, g, q)$. Then there exists an orthogonal $q$-basis $\{x, qx, q^{2}x, q^{3}x\}$ in $T_{p}M$.
\end{thm}

\begin{proof}
Let $\{x, qx, q^{2}x, q^{3}x\}$ be a $q$-basis in $T_{p}M$ of a manifold $(M, g, q)$. Then the triples of vectors $\{x, qx, q^{2}x\}$; $\{x, qx, q^{3}x\}$; $\{x, q^{2}x, q^{3}x\}$; $\{qx, q^{2}x, q^{3}x\}$ form four congruent pyramids. We consider for example one of them formed by $\{x, qx, q^{2}x\}$. The first side of it is isosceles triangle with angles $\angle(x,qx)=\varphi$, $\dfrac{\pi-\varphi}{2}$, $\dfrac{\pi-\varphi}{2}$. The second side of it is isosceles triangle with angles $\angle(qx,q^{2}x)=\varphi$, $\dfrac{\pi-\varphi}{2}$, $\dfrac{\pi-\varphi}{2}$. The third side is isosceles triangle with angles $\angle(x,q^{2}x)=\theta$, $\dfrac{\pi-\theta}{2}$, $\dfrac{\pi-\theta}{2}$. The fourth side is isosceles triangle with angles $\angle(x-qx,q^{2}x-qx)=\phi$, $\dfrac{\pi-\phi}{2}$ and $\dfrac{\pi-\phi}{2}$.
From the Cosine Rule applied to the fourth side and from \eqref{2.2} we get
$$2g(x, x)(1-\cos\theta)=4g(x, x)(1-\cos\varphi)\cos\phi,$$
and then
$$\cos \phi=\frac{1-2\cos\varphi+cos\theta}{2(1-cos\varphi)}.$$
From the above and $-1<\cos\phi<1$ we find
$$4\cos \varphi - \cos \theta < 3.$$
 The angles $\varphi =\dfrac{\pi}{2}$, $\theta=\dfrac{\pi}{2}$ satisfy the above inequality. Having in mind Theorem~\ref{t3.1} we prove that there exists an orthogonal $q$-basis in $T_{p}M$.
\end{proof}

\section{Curvature properties of $(M, g, q)$}\label{3}
Let $\nabla$ be the Riemannian connection of $g$ for a manifold $(M, g, q)$.
Let $R$ be the curvature tensor field of $\nabla$ of type $(0,4)$, and $R$ satisfies the identity
\begin{equation}\label{3.1}
R(x, y, qz, qu)=R(x, y, z, u).
\end{equation}

We note, that by identities like (\ref{3.1}) in \cite{1}, \cite{2} it have been defined the subclass of almost complex manifolds with Norden metric and the subclass of almost Hermitian manifolds respectively.

The sectional curvature $\mu$ of $2$-plane $\{x,y\}$ from $T_{p}M$ is expressed by the formula \cite{5}
\begin{equation}\label{3.8}
\mu(x,y)=\frac{R(x,y,x,y)}{g(x,x)g(y,y)-g^{2}(x,y)}.
\end{equation}
\begin{thm}
Let $(M, g, q)$ be a manifold with property \eqref{3.1}. Let $x$ induce a $q$-basis in $T_{p}M$. Then for the sectional curvature $\mu$ of $2$-planes we have
\begin{equation}\label {3.2}
\mu(x,qx)=\mu(qx,q^{2}x)=\mu(q^{2}x,q^{3}x)=\mu(q^{3}x,x),
\end{equation}
\begin{equation}\label{3.3}
\mu(x,q^{2}x)=\mu(qx,q^{3}x)=0.
\end{equation}
\end{thm}
\begin{proof}
From (\ref{3.1}) we find
\begin{equation}\label{3.4}
R(x,y,z,u)=R(x,y,qz,qu)=R(x,y,q^{2}z,q^{2}u)=R(x,y,q^{3}z,q^{3}u).
\end{equation}
 In (\ref{3.4}) we substitute

  1)\ $u$ for $qx$, $y$ for $qx$ and $z$ for $x$;

   2)\ $z$ for $x$, $y$ for $q^{2}x$ and $u$ for $q^{2}x$;

   3)\ $z$ for $x$, $y$ for $q^{3}x$ and $u$ for $q^{3}x$\\
and obtain respectively
\begin{equation}\label{3.5}
R(x,qx,x,qx)=R(x,qx,qx,q^{2}x)=R(x,qx,q^{2}x,q^{3}x)=R(x,qx,q^{3}x,qx),
\end{equation}
\begin{equation}\label{3.6}
R(x,q^{2}x,x,q^{2}x)=R(x,q^{2}x,qx,q^{3}x)=R(x,q^{2}x,q^{2}x,x)=R(x,q^{2}x,q^{3}x,x),
\end{equation}
\begin{equation}\label{3.7}
R(x,q^{3}x,x,q^{3}x)=R(x,q^{3}x,qx,x)=R(x,q^{3}x,q^{2}x,x)=R(x,q^{3}x,q^{3}x,q^{2}x).
\end{equation}
Using (\ref{3.5}), (\ref{3.7}) and (\ref{3.8}) we get (\ref{3.2}) and using (\ref{3.6}) and (\ref{3.8}) we get (\ref{3.3}).
\end{proof}
We see that every $2$-plane $\{x,qx\} \in T_{p}M$ has only two $q$-bases $\{x,qx\}$ or $\{-x,-qx\}$.
So the sectional curvature $\mu$ of $\{x,qx\}$ is a function of the $\angle(x,qx)=\varphi$, i.e. $\mu(x,qx)=\mu(\varphi)$.
\begin{prop} Let $(M, g, q)$ be a manifold with property \eqref{3.1} and $u$ induce a $q$-basis in $T_{p}M$.
If $\{x,qx,q^{2}x,q^{3}x\}$ is an orthonormal $q$-basis in $T_{p}M$, then the sectional curvature satisfies
\begin{equation}\label{3.9}
\mu(\varphi)=\dfrac{1}{1-\cos^{2}\varphi}\mu(\dfrac{\pi}{2}),
\end{equation}
where $\varphi=\angle(u,qu)$.
\end{prop}
\begin{proof}
Let $u=\alpha x+\beta qx +\gamma q^{2}x+\delta q^{3}x$, where $\alpha,\beta,\gamma, \delta \in \mathbb{R}$. Then $qu=\delta x+\alpha qx +\beta q^{2}x + \gamma q^{3}x$, $q^{2}u=\gamma x+\delta qx +\alpha q^{2}x +\beta q^{3}x$ and $q^{3}u=\beta x+\gamma qx +\delta q^{2}x + \alpha q^{3}x$.
We calculate
\begin{equation}\label{tr1}
\cos\varphi=\alpha \beta + \alpha\delta + \beta \gamma + \delta\gamma; \qquad
\cos\theta=2\alpha\gamma+2\beta\delta,
\end{equation}
where $\theta=\angle(u,q^{2}u)$.
Then using the linear properties of the curvature tensor $R$ and having in mind (\ref{3.5})--(\ref{3.7}), we obtain
\begin{equation}\label{tr2}
\begin{split}
R(u, qu, u, qu)&=\Big((\alpha^{2}+\gamma^{2}-2\beta\delta)^{2}+(\beta^{2}+\delta^{2}-2\gamma\alpha)^{2}\\&+2(\alpha^{2}+\gamma^{2}-2\beta\delta)(\beta^{2}+\delta^{2}-2\gamma\alpha)\Big)R(x, qx, x, qx).
\end{split}
\end{equation}
From \eqref{tr1} we get
\begin{equation}\label{tr3}
\begin{split}
(1-\cos\theta)^{2}R(u, qu, u, qu)&=\Big((\alpha^{2}+\gamma^{2}-2\beta\delta)^{2}+(\beta^{2}+\delta^{2}-2\gamma\alpha)^{2}\\&+2(\alpha^{2}+\gamma^{2}-2\beta\delta)(\beta^{2}+\delta^{2}-2\gamma\alpha)\Big)R(x, qx, x, qx)
\end{split}
\end{equation}
We substitute (\ref{tr2}) and \eqref{tr3} in (\ref{3.8}) and obtain (\ref{3.9}).
\end{proof}

\section{Parallelity of the circulant structure $q$}\label{4}

\begin{thm}
Let $\nabla$ be the Riemannian connection of $g$ of a manifold $(M, g, q)$. Then the structure $q$ is parallel with respect to the Riemannian connection $\nabla$ if and only if
\begin{equation}\label{f4}
    \grad A=(\grad C)q^{2}\ ,\quad 2\grad B=(\grad C)(q+q^{3}),
\end{equation}
where $\grad A$, $\grad B$ and $\grad C$ are gradients of the functions $A$, $B$ and $C$.
\end{thm}

\begin{proof}
Let the structure $q$ be parallel with respect to the Riemannian connection $\nabla$ of a manifold $(M, g, q)$, i.e. $\nabla q=0$.
Let $\Gamma_{ij}^{s}$ be the Christoffel symbols of $\nabla$. If
$\nabla q=0$, then
\begin{equation}\label{f5}
\nabla_{i}q^{s}_{j}=\partial_{i}q^{s}_{j}+\Gamma_{ik}^{s}q^{k}_{j}-\Gamma_{ij}^{k}q^{s}_{k}=0.
\end{equation}
From (\ref{f2}) and (\ref{f5}) we get
\begin{equation}\label{f6}
\Gamma_{ik}^{s}q^{k}_{j}=\Gamma_{ij}^{k}q^{s}_{k}.
\end{equation}
We denote
\begin{equation}\label{31}
    A_{i}=\dfrac{\partial A}{\partial X^{i}}\ ,\quad B_{i}=\dfrac{\partial B}{\partial X^{i}}\ ,\quad C_{i}=\dfrac{\partial C}{\partial X^{i}}\ ,
\end{equation}
where $A$, $B$ and $C$ are the functions from (\ref{f1}).

We find the inverse matrix of $(g_{ij})$ as follows:
\begin{equation}\label{f3}
    (g^{ij})=\frac{1}{D}\begin{pmatrix}
      \overline{A} & \overline{B} & \overline{C} & \overline{B} \\
      \overline{B} & \overline{A} & \overline{B} & \overline{C}\\
      \overline{C} & \overline{B} & \overline{A} &  \overline{B} \\
      \overline{B} & \overline{C} & \overline{B} & \overline{A}\\
    \end{pmatrix},\qquad D=(A-C)((A+C)^{2}-4B^{2}),
\end{equation}
where $\overline{A}=A(A+C)-2B^{2}$, $\overline{B}=B(C-A)$,
$\overline{C}=2B^{2}-C(A+C)$.

Using (\ref{f1}), (\ref{f2}), (\ref{f6})--(\ref{f3})
and the well known identities
\begin{equation}\label{2.3}
2\Gamma_{ij}^{s}=g^{as}(\partial_{i}g_{aj}+\partial_{j}g_{ai}-\partial_{a}g_{ij}),
\end{equation}
 after a long computation we get the following system:
 \begin{align*}
  &A_{4}-B_{1}+B_{3}-C_{2}=0,\\ &A_{4}+B_{1}-B_{3}-C_{2}=0,\\
 &2A_{2}+A_{4}-3B_{1}-B_{3}+C_{2}=0,\\
&A_{3}+B_{2}-B_{4}-C_{1}=0,\\
 &A_{3}-B_{2}+B_{4}-C_{1}=0,\\ &A_{2}-B_{1}+B_{3}-C_{4}=0,\\
 &A_{2}+B_{1}-B_{3}-C_{4}=0,\\ &A_{4}-B_{1}+3B_{3}+C_{2}+2C_{4}=0,\\ &A_{2}+2A_{4}-3B_{1}-B_{3}+C_{4}=0,\\
&A_{2}+2A_{4}-B_{1}-3B_{3}+C_{4}=0,\\ &A_{1}+2A_{3}-3B_{2}-B_{4}+C_{3}=0,\\ &A_{1}-B_{2}+B_{4}-C_{3}=0,\\
 &A_{3}-B_{2}-3B_{4}+C_{1}+2C_{3}=0,\\ &A_{1}-B_{2}-3B_{4}+2C_{1}+C_{3}=0,\\ &2A_{1}+A_{3}-B_{2}-3B_{4}+C_{1}=0,\\
 &A_{2}-B_{1}-3B_{3}+2C_{2}+C_{4}=0.
\end{align*}
The last system implies
 \begin{align}\label{system2}\nonumber
   &A_{1}=C_{3},\  A_{2}=C_{4},\  A_{3}=C_{1},\  A_{4}=C_{2},\  B_{1}=B_{3},\\&B_{2}=B_{4},\  2B_{1}=C_{4}+C_{2},\ 2B_{2}=C_{1}+C_{3}.
\end{align}
 Then we obtain that (\ref{f4}) is valid.

Inversely, let (\ref{f4}) be valid. We can verify that (\ref{system2})
is valid, too. The identities (\ref{system2}) imply (\ref{f6}) and
consequently (\ref{f5}) is true. So $\nabla q=0$.
\end{proof}

\begin{prop} Let $(M, g, q)$ be a manifold with parallel structure $q$ with respect of $g$. Then $(M, g, q)$ is a manifold with property \eqref{3.1}.
\end{prop}

\begin{proof}
The condition $\nabla q=0$ implies $\nabla_{i}q_{s}^{j}=0$. The integrability condition of this system is
\begin{equation}\label{f7}
R^{a}_{jkl}q_{a}^{s}=R^{s}_{akl}q_{j}^{a},
\end{equation}
where $R^{a}_{jkl}$ are the local  coordinates of $R$.
From (\ref{f7}) we find
\begin{equation}\label{f8}
R_{ajkl}q^{a}_{.s}=R_{sakl}q^{a}_{j.}.
\end{equation}
We get $q^{a}_{.s}$ are the local coordinates of $q^{3}$. So (\ref{f8}) implies
$$R(q^{3}u,v,w,t)=R(u,qv,w,t)$$
from which (\ref{3.1}) follows. Then $(M, g, q)$ has the property \eqref{3.1}.
\end{proof}

\vspace{6pt}
\author{Dimitar Razpopov\\ Department of Mathematics and Physics\\ Agricultural University of Plovdiv\\
Bulgaria 4000\\
e-mail:dimitrerazpopov@hotmail.com}
% ------------------------------------------------------------------------------------------------------------------------
\end{document}